\documentclass[12pt]{article}%
\usepackage{amssymb}
\usepackage{amsmath}
\usepackage{amsfonts}
\usepackage{graphicx}
\usepackage{color}
\usepackage{verbatim}
\usepackage{geometry}%
\providecommand{\U}[1]{\protect\rule{.1in}{.1in}}
\providecommand{\U}[1]{\protect\rule{.1in}{.1in}}
\newtheorem{theorem}{Theorem}[section]
\newtheorem{acknowledgement}[theorem]{Acknowledgement}

\newtheorem{corollary}[theorem]{Corollary}

\newtheorem{definition}[theorem]{Definition}
\newtheorem{example}[theorem]{Example}

\newtheorem{lemma}[theorem]{Lemma}

\newtheorem{proposition}[theorem]{Proposition}
\newtheorem{remark}[theorem]{Remark}

\newenvironment{proof}[1][Proof]{\noindent\textbf{#1.} }{\ \rule{0.5em}{0.5em}}
\begin{document}

\title{Tail estimates for Markovian rough paths }
\author{Thomas Cass and Marcel Ogrodnik}
\maketitle

\begin{abstract}
We work in the context of Markovian rough paths associated to a class of
uniformly subelliptic Dirichlet forms (\cite{FrizVictoir2008}) and prove a
better-than-exponential tail estimate for the accumulated local $p$-variation
functional, which has been introduced and studied in \cite{CLL2013}. We
comment on the significance of these estimates to a range of currently-studied
problems, including the recent results of Ni Hao \cite{Hao2014}, and Chevyrev
and Lyons \cite{ChevyrevLyons2014}.

\end{abstract}

\parindent 0pt

\section{Introduction}

Lyons's rough path theory has allowed a pathwise interpretation to be given to
stochastic differential equations of the form%

\[
dY_{t}=V\left(  Y_{t}\right)  dX_{t},\text{ }Y_{0}=y_{0},
\]
where the vector fields $V=\left(  V^{1},...,V^{d}\right)  $ are driven along
an $%
%TCIMACRO{\U{211d} }%
%BeginExpansion
\mathbb{R}
%EndExpansion
^{d}$-valued rough random signal $X$. An important feature of Lyons' approach
-- as compared, say, to the classical framework of It\^{o} -- is the
relaxation of the condition that $X$ be a semimartingale. There is typically
no way of accommodating this feature within It\^{o}'s or any comparable
theory. Furthermore there are fundamental classes of random signals where the
semimartingale property is either absent, or only present in special cases,
e.g. Markov processes, fractional Brownian motions and, more broadly, the
family of Gaussian processes. Study of the Gaussian-driven RDEs using rough
path analysis has been especially prolific over recent years, we reference
\cite{FrizVictoir2010a}, \cite{CassFriz2010}, \cite{HairerPillai2013},
\cite{FrizRiedel2014}, \cite{Inahama2014}, \cite{Riedel2014} and
\cite{BaudoinOuyang2014} as an illustrative, although by no means exhaustive,
list of applications.

\qquad Semimartingales have a well-defined quadratic variation process. It is
widely appreciated, at least for continuous semimartingales, that control of
the quadratic variation provides insight on the moments, tails and deviations
of the semimartingale itself. The exponential martingale inequality (see,
e.g., \cite{RogersWilliams2000a}) and the Burkholder-Davis-Gundy inequalities
(see, e.g., \cite{Burkholder1966}) are prime examples of this principle in
practice. The latter result in particular, allows one to control the moments
of linear differential equations%
\[
dY_{t}=AY_{t}dX_{t},Y_{0}=y_{0},
\]
where $A$ is in Hom$\left(
%TCIMACRO{\U{211d} }%
%BeginExpansion
\mathbb{R}
%EndExpansion
^{d},%
%TCIMACRO{\U{211d} }%
%BeginExpansion
\mathbb{R}
%EndExpansion
^{e}\right)  $ and $X$ is a semimartingale. A more sophisticated example to
which this idea applies is the case when $Y$ is the derivative of the flow of
an SDE, which is well known to solve an SDE with linear growth vector fields.
In many applications, such as Malliavin calculus, it is crucial to show that
this derivative process (and its inverse) has finite moments of all orders.

\qquad In rough path theory, by contrast, one deliberately postpones using
probabilistic features of $X$. Indeed, a key advantage is the separation
between the deterministic theory, which is used to solve the differential
equation, and the probability, which is used to enhance the driving path to a
rough path. This separation however can -- and often, does -- introduce
complications in probabilistic applications. For instance, in trying to prove
moment estimates of the type discussed in the last paragraph using a rough
path approach, it is reasonable to try to integrate the natural growth
estimate for the solution, which in this case has the form (see
\cite{FrizVictoir2010})%

\begin{equation}
\left\vert \left\vert \mathbf{Y}\right\vert \right\vert _{p\text{-}var,\left[
0,T\right]  }\leq C\exp\left(  C\left\vert \left\vert \mathbf{X}\right\vert
\right\vert _{p\text{-}var,\left[  0,T\right]  }^{p}\right)  ,
\label{crude bound}%
\end{equation}

where $\left\vert \left\vert \mathbf{X}\right\vert \right\vert _{p\text{-}%
var,\left[  0,T\right]  }$ denotes $p-$variation of the rough path enhancement
of $X$. In the case when $X$ is the Gaussian (even Brownian) rough path, this
inequality is useless for proving moment estimates because the right-hand side
is not integrable; $\left\vert \left\vert \mathbf{X}\right\vert \right\vert
_{p\text{-}var,\left[  0,T\right]  }$ has only Gaussian tail. Nevertheless, it
is possible to surmount this problem, as demonstrated by \cite{CLL2013}. The
key idea is to use a slight sharpening of the estimate (\ref{crude bound}) to
one of the form (see \cite{CLL2013})%

\begin{equation}
\left\vert \left\vert \mathbf{Y}\right\vert \right\vert _{p\text{-}var,\left[
0,T\right]  }\leq C\exp\left(  \sup_{\overset{D=(t_{i})}{\left\vert \left\vert
\mathbf{X}\right\vert \right\vert _{p\text{-}var,\left[  t_{i},t_{i+1}\right]
}\leq1}}\sum_{i:t_{i}\in D}\left\vert \left\vert \mathbf{X}\right\vert
\right\vert _{p\text{-}var,\left[  t_{i},t_{i+1}\right]  }^{p}\right)
:=C\exp\left[  CM(\mathbf{X},[0,T])\right]  . \label{CLL}%
\end{equation}

The functional $M$ is called the accumulated local $p$-variation. While it may
not appear on first inspection that this estimate helps much, in fact it
considerably improves the tail analysis mentioned above. The main result of
\cite{CLL2013} is the following tail estimate for Gaussian rough paths
$\mathbf{X}$:%

\begin{equation}
\mathbb{P}\left(  M(\mathbf{X},[0,T])>x\right)  \leq\exp\left(  -cx^{2/q}%
\right)  , \label{estimate}%
\end{equation}

where $q\in\lbrack1,2)$ is a parameter related to the Cameron-Martin Hilbert
space of $X$. It follows as a consequence that the left-hand side of
(\ref{CLL}) has moments of all orders.

\qquad The strategy for proving the estimate (\ref{estimate}) in the Gaussian
setting is somewhat subtle. The first step is to introduce the so-called
$p$-variation greedy partition by setting%

\[
\tau_{0}=0\text{, and }\tau_{n+1}=\inf\left\{  t\geq\tau_{n}:\left\vert
\left\vert \mathbf{x}\right\vert \right\vert _{p\text{-}var,\left[  \tau
_{n},t\right]  }=1\right\}  \wedge T.
\]

An integer-valued random variable defined by%

\begin{equation}
N_{p\text{-var}}(\mathbf{x},[0,T])=\sup\left\{  n\in\mathbb{N}\cup\left\{
0\right\}  :\tau_{n}<T\right\}  \label{n pvar}%
\end{equation}

then counts the number of distinct intervals in the partition $\left(
\tau_{n}\right)  _{n=0}^{\infty}$. Second, a relatively simple argument gives%

\[
N_{p\text{-var}}(\mathbf{x},[0,T])\leq M(\mathbf{x},[0,T])\leq2N_{p\text{-var}%
}(\mathbf{x},[0,T])+1,
\]

and hence the tail of the random variable $M(\mathbf{X},[0,T])$ can be deduced
from that of $N_{p\text{-var}}(\mathbf{X},[0,T])$. Third, the estimate
(\ref{estimate}) is proved for $N_{p\text{-var}}(\mathbf{X},[0,T])$ in place
of $M(\mathbf{X},[0,T]);$ the two key tools in doing this are (Borell's)
Gaussian isoperimetric inequality (see, e.g., \cite{Borell1975},
\cite{BakryLedoux1996}), and the Cameron-Martin embedding theorem of
\cite{FrizVictoir2010}.

\qquad In this paper we study this problem for a different class of rough
paths: the Markovian rough paths. Rough paths which are themselves Markov, or
which are the lifts of such processes, have been studied previously. In
\cite{BassHamblyLyons2002}, for example, the authors start with an a
reversible $%
%TCIMACRO{\U{211d} }%
%BeginExpansion
\mathbb{R}
%EndExpansion
^{d}$-valued continuous Markov process $X$ having a stationary probability
measure $\mu.$ By assuming a moment condition on the increments of $X$ and by
starting $X$ \ in its stationary distribution, they construct a L\'{e}vy-area
process as a limit of dyadic piecewise linear approximations to $X$. The
argument uses a forward-backward martingale decomposition, in the spirit of
\cite{LyonsZhang1994}, which is applied to a natural sequence of
approximations to the area. The reversibility of $X$ and the anti-symmetry of
the L\'{e}vy-area are used in an attractive way to realise suitable
cancellations in this approximating sequence. Earlier work by Lyons and Stoica
(see \cite{LyonsStoica1999}) has also exploited the forward-backward
martingale decomposition in the construction of the L\'{e}vy-area. An
alternative approach, which we will follow most closely in our presentation,
was proposed in \cite{FrizVictoir2008} and \cite{FrizVictoir2010}. Here
$\mathbf{X=}\left(  X,A\right)  $ is constructed not by \textit{enhancing} $X$
as in the initially mentioned approach, but directly as the Markov process
associated with (the Friedrich's extension of) a Dirichlet form (see Section
\ref{markov} for a review of this idea).

\qquad There are big obstacles to implementing the Gaussian approach of
\cite{CLL2013} in this setting. The most important is the lack of a usable
substitute for the isoperimetric inequality and, relatedly, the Cameron-Martin
embedding theorem (indeed, there is no longer any Cameron-Martin space!).
Analogous results which exist in the literature (e.g. \cite{BakryLedoux1996},
\cite{CapitaineHsuLedoux1997}) do not seem easy to implement here. As a
consequence we have to re-think the whole strategy upon which \cite{CLL2013}
is founded. In so doing we gain important insights into the general principles
for proving estimates of the type (\ref{estimate}). In summary, these are:

\begin{enumerate}
\item That it can be useful to determine the greedy partition $\left(
\sigma_{n}\right)  _{n=0}^{\infty}$ from a metric topology which is weaker
than the $p$-variation rough path topology. Let $d$ denote the metric, and
$N_{d}(\mathbf{x},[0,T])$ the integer corresponding to the greedy partition
under this metric. Then, clearly, $N_{d}(\mathbf{x},[0,T])\leq N_{p\text{-var}%
}(\mathbf{x},[0,T]).$ This has the immediate advantage of making the proof of
the tail estimate for $N_{d}(\mathbf{X},[0,T])$ easier to prove than for
$N_{p\text{-var}}(\mathbf{X},[0,T])$. The price one pays is that it is no
longer true that%
\[
\left\vert \left\vert \mathbf{X}\right\vert \right\vert _{p\text{-}var,\left[
\sigma_{n},\sigma_{n+1}\right]  }\leq1\text{ for all }n=0,1,2...
\]
Nevertheless, the control of $\mathbf{X}$ in some topology -- even a weaker
one than $p$-variation---is often sufficient to dramatically improve the tail
behaviour of the random variable $\left\vert \left\vert \mathbf{X}\right\vert
\right\vert _{p\text{-}var,\left[  \sigma_{n},\sigma_{n+1}\right]  }$. Similar
observations to this have been made before in other contexts, e.g.
\cite{LyonsZeitouni1999} and in support theorems, \cite{Lyons2004},
\cite{BenArousGradinaruLedoux1994}, \cite{FrizLyonsStroock2006}.

\item A natural choice of metric in the Markovian regime of this paper is the
supremum-metric for rough paths, and this is likely to be so for other classes
of random rough paths, too. In the present setting, we can control the tails
on $N$ using a combination of large deviations estimates, Gaussian heat kernel
bounds and exponential Tauberian theorems. For other examples, a different way
of obtaining these bounds will be needed. But the study of tail estimates for
the maximum of a stochastic process is a much more widely addressed subject
than the corresponding study for $p$-variation, see e.g. \cite{Talagrand2014}.
There are likely to be many more examples which can be approached by adapting
these methods.
\end{enumerate}

\qquad\ We have already mentioned some applications. Without giving an
exhaustive list, or trying to anticipate all future uses of this work, we
briefly summarise what we believe will be the most immediately obvious sources
of impact. The chief application of \cite{CLL2013} has been in Gaussian
H\"{o}rmander theory to prove, for example, smoothness and other properties of
the density for Gaussian RDEs (see, e.g., \cite{HairerPillai2013},
\cite{Baudoin2013}, \cite{Baudoin2013a}, \cite{Baudoin2014a},
\cite{CassHairerLittererTindel2014}). A similar approach might be attempted
with Markovian signals, but one has to be careful -- unlike in the Gaussian
setting, the driving Markov process will no longer have a smooth density in
general. Nevertheless it is interesting to consider whether the It\^{o} map
preserves the density (and its derivatives -- if it has any) under
H\"{o}rmander's condition. Here the Malliavin method will radically break
down; abstract Wiener analysis will need to be replaced by analysis of the
Dirichlet form. \ 

\qquad Second, growth estimates involving the accumulated $p$-variation occur
naturally and generically in rough path theory; see \cite{FrizRiedel2013} for
a range of examples. We therefore expect uses of our results to be widespread.
In \cite{Cass2013} it was observed that $M(\mathbf{X},[0,T])$ appears in
optimal Lipschitz-estimates on the rough path distance between two different
RDE solutions. This has uses in fixed-point arguments, e.g. in studying
interacting McKean-Vlasov-type RDEs.

\qquad Another illustration of the use of our result can be found in very
interesting recent papers \cite{Hao2014} and \cite{ChevyrevLyons2014}. In
these papers the authors prove criteria for the law of a geometric rough path
to be determined by its expected signature. These criteria are formulated in
terms of the power series
\begin{equation}
\sum_{n=1}^{\infty}\lambda^{n}\left\vert \mathbb{E}\left[  \mathbf{X}%
_{0,T}^{n}\right]  \right\vert _{n}, \label{series}%
\end{equation}
where $S(\mathbf{X})_{0,T}=\sum_{k=0}^{\infty}\mathbf{X}_{0,T}^{k}$ denotes
the signature of a geometric rough path $\mathbf{X,}$ and $\left\vert
\mathbb{\cdot}\right\vert _{n}$ is a suitable norm on $\left(
%TCIMACRO{\U{211d} }%
%BeginExpansion
\mathbb{R}
%EndExpansion
^{d}\right)  ^{\otimes n}$. An important result in \cite{ChevyrevLyons2014} is
that the radius of convergence of (\ref{series}) being infinite is sufficient
for $\mathbb{E}[S(\mathbf{X})_{0,T}]$ to determine the law of $\mathbf{X}$
uniquely over $[0,T]$. \ The work of Ni Hao \cite{Hao2014} and Friz and Riedel
\cite{FrizRiedel2012} complements this result by proving an upper bound on the
signature $S(\mathbf{X})_{0,T}$ in terms of $N_{p\text{-var}}(\mathbf{X}%
,[0,T])$. In \cite{Hao2014} these estimates are then used to show that if
$N_{p\text{-var}}(\mathbf{X},[0,T])$ \ has a Gaussian tail then the radius of
convergence of the series (\ref{series}) is infinite. In
\cite{ChevyrevLyons2014}, this statement is refined to show that any
better-than-exponential tail of $N_{p\text{-var}}(\mathbf{X},[0,T])$ suffices
for the same conclusion, and that a somewhat weaker determination of the law
of $\mathbf{X}$\ is possible when the tail is only exponential. One example
cited in \cite{ChevyrevLyons2014} is the class of Markovian rough paths
stopped on leaving a domain (the domain is required to have some boundedness
properties in order for it to have a well-defined diameter). For this class
they are able to show exponential integrability of $N_{p\text{-var}%
}(\mathbf{X},[0,T])$. Our main result, Theorem \ref{main thm}, imposes no
restriction on the domain of $\mathbf{X}$ and we prove a stronger
tail-estimate; more exactly, we prove one which is better-than-exponential in the
sense that%

\[
\mathbb{P}\left(  M\left(  \mathbf{X};[0,T]\right)  >R\right)  \leq
C\exp\left(  -CR^{2(1-1/p)}\right)  \text{ for any }p>2,\text{ where }C=C_{p}.
\]
This is obviously better than just exponential decay, and it has consequence
one may verify the stronger criterion mentioned in the above work.\ One
immediate application of our results therefore is to broaden substantially the
range of examples to which the results of Chevyrev-Lyons-Ni Hao
\cite{ChevyrevLyons2014}, \cite{Hao2014} are known to apply.

\qquad The outline of the article is as follows. In Section \ref{rough path},
we give a general overview of the results of rough path theory required for
our analysis. In Section \ref{markov} we review the theory and key results for
Markovian rough paths. In Section \ref{ld} we use large deviations techniques
and exponential Tauberian theorems to prove that under the supremum metric the
integer associated to the greedy partition of a Markovian rough path has a
Gaussian tail. The main work is done in Section \ref{main}, where we prove a
crucial bound on the accumulated local $p$-variation in terms of the
aforementioned integer of the greedy partition and the accumulated
$p$-variation of the Markovian rough path between the points of this
partition. This result, in concert with heat kernel estimates and the results
of section \ref{markov}, allows us to prove our main theorem, i.e., the
accumulated local $p$-variation of a Markovian rough path has better-than-exponential tails.

\begin{acknowledgement}
The research of the first-named author is supported by EPSRC grant
EP/M00516X/1. The first-named author is also grateful to Bruce Driver for
conversations related to this work.
\end{acknowledgement}

\section{Rough paths\label{rough path}}

There are now many texts which outline the core content of rough path theory
(e.g., \cite{Lyons1998}, \cite{Lyons2007}, \cite{FrizVictoir2010} and
\cite{FrizHairer2014}), here we focus on gathering together relevant notation.

To start with, assume $V$ is a $d$-dimensional real vector space. Then a basic
role in the theory is played by the truncated tensor algebra which for $N\in%
%TCIMACRO{\U{2115} }%
%BeginExpansion
\mathbb{N}
%EndExpansion
$ is the set%

\[
T^{N}\left(  V\right)  :=\left\{  g=\left(  g^{0},g^{1},...,g^{N}\right)
:g^{k}\in V^{\otimes k},k=0,1,...,N\right\}
\]

equipped with the truncated tensor product. Two subsets of $T^{N}\left(
V\right)  $ of particular interest are%

\[
\tilde{T}:=\tilde{T}^{N}\left(  V\right)  :=\left\{  h\in T^{N}\left(
V\right)  :g^{0}=1\right\}  \text{ and }\mathfrak{\tilde{t}:=\tilde{t}}%
^{N}(V):=\left\{  A\in T^{N}\left(  V\right)  :A^{0}=0\right\}  .
\]

It is easy to see that $\tilde{T}$ is a group under truncated tensor
multiplication. In fact it is a Lie group and the vector space
$\mathfrak{\tilde{t}}$ \ is its Lie algebra Lie$\left(  \tilde{T}\right)  $,
i.e. $\mathfrak{\tilde{t}}$ is tangent space to $\tilde{T}$ at the group
identity $1$. The diffeomorphisms $\log:\tilde{T}\rightarrow$
$\mathfrak{\tilde{t}}$ $\ $\ and $\exp:\mathfrak{\tilde{t}}\rightarrow
\tilde{T}$ defined respectively by the power series
\[
\log\left(  g\right)  =\sum_{k=1}^{N}\frac{\left(  -1\right)  ^{k-1}}%
{k}\left(  g-1\right)  ^{k}\text{ and }\exp\left(  A\right)  =\sum_{k=0}%
^{N}\frac{1}{k!}A^{k}%
\]
are mutually inverse, and $\log$ defines a global chart on $\tilde{T}.$ The
map $\exp$ coincides with the Lie group exponential, i.e. for every $A,$
$\exp\left(  A\right)  =\gamma_{A}\left(  1\right)  $ where $\gamma_{A}:%
%TCIMACRO{\U{211d} }%
%BeginExpansion
\mathbb{R}
%EndExpansion
\rightarrow$ $\tilde{T}$ is the unique integral curve through the identity of
the left-invariant vector field associated with $A.$

\qquad In the paper it will be useful to realise the group structure of
$\tilde{T}$ on the set $\mathfrak{\tilde{t}}$ $\mathfrak{.}$ To do this we
define a product $\ast:$ $\mathfrak{\tilde{t}}$ $\mathfrak{\times\tilde
{t}\rightarrow\tilde{t}}$ using the functions $\exp$ and $\log$ as follows%
\[
A\ast B:=\log\left(  \exp\left(  A\right)  \exp\left(  B\right)  \right)
\text{ for all }A,B\in\mathfrak{\tilde{t}.}%
\]
Under this definition $\left(  \mathfrak{\tilde{t},\ast}\right)  $ is again a
Lie group with identity element $0,$ and $\exp$ is then a Lie group
isomorphism from $\left(  \mathfrak{\tilde{t},\ast}\right)  $ to $\tilde{T}.$
The differential of $\exp$ at $0$ then pushes forward tangent vectors in
$T_{0}\mathfrak{\tilde{t}}$ to elements of the vector space $\mathfrak{\tilde
{t}.}$ This linear isomorphism is easily seen to be the identity map on
$\mathfrak{\tilde{t},}$ hence Lie$\left(  \mathfrak{\tilde{t},\ast}\right)
=\mathfrak{\tilde{t}}$ as a vector space. The Lie group exponential map
Lie$\left(  \mathfrak{\tilde{t},\ast}\right)  \rightarrow\left(
\mathfrak{\tilde{t},\ast}\right)  $ \ also equals the identity map on
$\mathfrak{\tilde{t}}$, and the Campbell-Baker-Hausdorff formula (see
\cite{Dynkin1947}, \cite{Strichartz1987}) can be used to show that the Lie
bracket induced by $\left(  \mathfrak{\tilde{t},\ast}\right)  $ agrees with
$AB-BA,$ the commutator Lie bracket derived from the original truncated tensor multiplication.

\qquad We let $\mathfrak{g}^{N}\mathfrak{:=g=}$Lie$\left(  V\right)  $ be the
Lie algebra generated by $V.$\ The vector space $\mathfrak{g}$ is an embedded
submanifold of $\ \mathfrak{\tilde{t}}$ and is also a subgroup of $\left(
\mathfrak{\tilde{t},\ast}\right)  $ under the product $\mathfrak{\ast.}$ It
follows that $\left(  \mathfrak{g,\ast}\right)  $ is a Lie group, which we
call the step-$N$ nilpotent Lie group with $d$ generators. The Lie algebra
associated with $\left(  \mathfrak{g,\ast}\right)  $ is the vector space
$\mathfrak{g.}$

\begin{definition}
For any $a\in V$ we define $B_{a}$ to be the unique left-invariant vector
field on $\left(  \mathfrak{g,\ast}\right)  $ associated with $\left(
0,a,0,...,0\right)  \in\mathfrak{g}.$ Given $A\in\mathfrak{g}$ we then define
the horizontal subspace $\mathcal{H}_{A}$ at $A\in\mathfrak{g}$ to be the
vector subspace of $\mathfrak{g}$ given by
\[
\mathcal{H}_{A}=span\left\{  B_{a}\left(  A\right)  :a\in V\right\}  .
\]
An absolutely continuous curve $\gamma:\left[  0,T\right]  \rightarrow
\mathfrak{g}$ is then said to be horizontal if $\dot{\gamma}\left(  t\right)
\in\mathcal{H}_{\gamma\left(  t\right)  }$ for almost every $t\in\left[
0,T\right]  .$
\end{definition}

\begin{remark}
For example when $N=2$ a simple calculation shows that%
\[
B_{a}\left(  A\right)  =a+\frac{1}{2}\left[  A^{1},a\right]  ,\text{ where
}A=\left(  A^{1},A^{2}\right)  .
\]

\end{remark}

We will equip $V$ with a norm and consider paths $x\ $belonging to
$C^{1\text{-var}}\left(  [0,T],V\right)  ,$ the space of continuous $V$-valued
paths of finite 1-variation $\Vert\mathbf{x}\Vert_{1-\text{var};\left[
0,T\right]  }$. The truncated signature $S_{N}(x)$ of $x$ is defined by
\[
S_{N}(x)_{0,\cdot}:=1+\sum_{k=1}^{N}\int_{0<t_{1}<....<t_{k}<\cdot}dx_{t_{1}%
}\otimes...\otimes dx_{t_{k}}=:1+\sum_{k=1}^{N}\mathbf{x}_{0,\cdot}^{k}%
\in\tilde{T}^{N}\left(  V\right)  .
\]
It is well-known (see \cite{FrizVictoir2010}) that $\log S_{N}(x)_{0,\cdot}$
is a path which takes values in the group $\left(  \mathfrak{g,\ast}\right)
.$ Any horizontal curve starting from the $0,$ the identity in\ $\left(
\mathfrak{g,\ast}\right)  ,$ can be realised as the unique solution to%
\[
d\gamma_{t}=B_{dx_{t}}\left(  \gamma_{t}\right)  ,\text{ }\gamma_{0}=0.
\]
This so-called horizontal lift of $x$ is easily shown to equal $S_{N}%
(x)_{0,\cdot}$.

\qquad A classical theorem of Chow (see, e.g., \cite{Gromov1996},
\cite{Montgomery2002}) shows that any distinct points in $\mathfrak{g}$ can be
connected by a horizontal curve (which is smooth in the case $N=2)$. This
gives rise to the Carnot-Carath\'{e}odory norm on $\left(  \mathfrak{g,\ast
}\right)  $ as the associated geodesic distance
\begin{equation}
\Vert g\Vert_{CC}:=\inf\left\{  \Vert\mathbf{x}\Vert_{1-\text{var};\left[
0,T\right]  }:x\in C^{1\text{-var}}\left(  [0,T],V\right)  \text{ and }%
S_{N}\left(  x\right)  _{0,T}=g\right\}  . \label{cc norm}%
\end{equation}
The function $\Vert\cdot\Vert_{CC}$ has the property of being a
\textit{homogeneous norm on }$\left(  \mathfrak{g,\ast}\right)  $. By this we
mean a map $\Vert\cdot\Vert:\left(  \mathfrak{g,\ast}\right)  \rightarrow
\mathbb{R}_{\geq0}$ which vanishes at the identity and is homogeneous in the
sense that
\[
\Vert\delta_{r}g\Vert=\left\vert r\right\vert \Vert g\Vert\text{ for every
}r\in%
%TCIMACRO{\U{211d} }%
%BeginExpansion
\mathbb{R}
%EndExpansion
,
\]
wherein $\delta_{r}:\mathfrak{g\rightarrow g}$ is the restriction to
$\mathfrak{g}$ of the scaling operator $\delta_{r}:\tilde{T}^{N}\left(
V\right)  \mathfrak{\rightarrow}\tilde{T}^{N}\left(  V\right)  $ defined by%
\[
\delta_{r}:\left(  1,g^{1},g^{2},...,g^{N}\right)  \rightarrow\left(
1,rg^{1},r^{2}g^{2},...,r^{N}g^{N}\right)  .
\]
In finite dimensions it is a basic fact (\cite{FrizVictoir2010}) that all such
homogeneous norms are Lipschitz equivalent, and the subset of symmetric and
subadditive homogeneous norms gives rise to metrics on $\left(
\mathfrak{g,\ast}\right)  $. The one which we will use most often is the
left-invariant Carnot-Carath\'{e}odory metric $d_{CC}$ determined from
(\ref{cc norm}) by%

\[
d_{CC}(g,h)=\Vert g^{-1}\ast h\Vert_{CC},\quad g,h\in\mathfrak{g}.
\]

\qquad For any path $\mathbf{x}:[0,T]\rightarrow\left(  \mathfrak{g,\ast
}\right)  $ the group structure provides us with a natural notion of increment
given by $\mathbf{x}_{s,t}:=\mathbf{x}_{s}^{-1}\ast\mathbf{x}_{t}.$\ For each
$\alpha$ in $(0,1]$ and $p$ in $[1,\infty)$ we can then let $C^{\alpha
-\text{H\"{o}l}}\left(  [0,T],\mathfrak{g}\right)  $ and $C^{p-\text{var}%
}\left(  [0,T],\mathfrak{g}\right)  $ be the subsets of the continuous
$\mathfrak{g}$-valued paths such that the following, respectively, are finite
real numbers%

\begin{align}
\Vert\mathbf{x}\Vert_{\alpha-\text{H\"{o}l};[0,T]}  &  :=\sup
_{\substack{\lbrack s,t]\subseteq\lbrack0,T], \\s\neq t}}\frac{\Vert
\mathbf{x}_{s,t}\Vert_{CC}}{|t-s|^{\alpha}},\label{Holder_norm}\\
\Vert\mathbf{x}\Vert_{p-\text{var};[0,T]}  &  :=\left(  \sup_{D=(t_{j})}%
\sum_{j:t_{j}\in D}\Vert\mathbf{x}_{t_{j},t_{j+1}}\Vert_{CC}^{p}\right)
^{1/p}, \label{norm}%
\end{align}

where, in the latter, the supremum runs over all partitions $D$ of the
interval $[0,T]$.

\begin{definition}
\label{weakly geometric paths} For $p\geq1$ we let
\[
WG\Omega_{p}\left(  V\right)  :=WG\Omega_{p}\left(  \left[  0,T\right]
,V\right)  :=C^{p-\text{var}}\left(  [0,T],\mathfrak{g}^{\lfloor p\rfloor
}\right)  .
\]
We call $WG\Omega_{p}\left(  V\right)  $ the set of weakly\footnote{The prefix
\textit{weakly }here is really a misnomer; what are customarily called weakly
geometric rough paths really ought to be called geometric rough paths. We
persist with it for the sake of consistency with the literature.} geometric
$p$-rough paths.
\end{definition}

\begin{remark}
Note that $C^{1/p-\text{H\"{o}l}}\left(  [0,T],\mathfrak{g}^{\lfloor p\rfloor
}\right)  \subset WG\Omega_{p}\left(  V\right)  .$
\end{remark}

The definitions (\ref{Holder_norm}) and (\ref{norm}) can be easily extended
for any compact subset $I\subset\mathbb{R}$ by simply replacing $[0,T]$ by
$I$. We will also consider the case where $I=[0,\infty),$ by which we mean the following.

\begin{definition}
For $p\geq1$ we define $C^{p-\text{var}}\left(  [0,\infty),\mathfrak{g}%
\right)  $ to be the subset of the continuous $\mathfrak{g}$-valued paths,
$C\left(  [0,\infty),\mathfrak{g}\right)  $ as follows%

\[
C^{p-\text{var}}\left(  [0,\infty),\mathfrak{g}\right)  :=\left\{
\mathbf{x\in}C\left(  [0,\infty),\mathfrak{g}\right)  :\forall T\geq0,\left.
\mathbf{x}\right\vert _{T}\in C^{p-\text{var}}\left(  [0,T],\mathfrak{g}%
\right)  \right\}
\]
where $\left.  \mathbf{x}\right\vert _{T}$ denotes the restriction of a path
$\mathbf{x}$ on $[0,\infty)$ to one on $\left[  0,T\right]  $. We define
$C^{1/p-\text{H\"{o}l}}\left(  [0,\infty),\mathfrak{g}\right)  $ similarly.
\end{definition}

We will later need the fact that for $\mathbf{x}\in C^{p-\text{var}}\left(
[0,T],\mathfrak{g}\right)  $ the map
\begin{equation}
\omega_{\mathbf{x}}(s,t):=\Vert\mathbf{x}\Vert_{p-\text{var};[s,t]}^{p}
\label{control}%
\end{equation}
is a \textit{control}; by this we mean it is a continuous, non-negative,
super-additive function on the simplex $\Delta_{T}:=\{\left(  s,t\right)
\in\lbrack0,T]\times\left[  0,T\right]  :0\leq s\leq t\leq T\}$ which is zero
on the diagonal (see {\cite[p.80]{FrizVictoir2010}}).

\section{Markovian rough paths\label{markov}}

Throughout the rest of this paper we will work on the finite dimensional
vector space $V=\mathbb{R}^{d}$, equipped with the Euclidean norm. We will
exploit the rich analysis of self-adjoint subelliptic differential operators,
and the corresponding probabilistic study of symmetric Markov processes. The
most prominent references for our setting include \cite{BeurlingDeny1959},
\cite{Stroock1988}, \cite{Sturm1994}, \cite{Sturm1995}, \cite{Sturm1996a},
\cite{Fukushima2010a} and \cite{Ma1992}. \newline We will work with the
Dirichlet form which for smooth compactly supported functions $f,g\in
C_{c}^{\infty}\left(  \mathfrak{g}\right)  $ is defined by
\begin{equation}
\mathcal{E}^{a}(f,g)=\sum_{i,j=1}^{d}\int_{\mathfrak{g}}a^{ij}(h)B_{i}%
f(h)B_{j}g(h)\,dm\left(  h\right)  =:\int_{\mathfrak{g}}\Gamma^{a}(f,g)\,dm.
\label{Dirichlet_form}%
\end{equation}
Here $B_{i}=B_{e_{i}}$ is the unique left-invariant vector field on $\left(
\mathfrak{g,\ast}\right)  $ associated with the $i^{\text{th}}$ standard basis
vector $e_{i}\in%
%TCIMACRO{\U{211d} }%
%BeginExpansion
\mathbb{R}
%EndExpansion
^{d}\subseteq$Lie$\left(
%TCIMACRO{\U{211d} }%
%BeginExpansion
\mathbb{R}
%EndExpansion
^{d}\right)  $; $m$ denotes the (bi-invariant) Haar measure on $\left(
\mathfrak{g,\ast}\right)  $, which coincides with the Lebesgue measure on the
vector space $\mathfrak{g}$; $a=\left(  a^{ij}\right)  _{i,j\in\left\{
1,...,d\right\}  }$ is a measurable map from $\mathfrak{g}$ \ to
$\mathcal{S}_{d}$, the space of $d\times d$ symmetric matrices$;$ and,
$\Gamma^{a}$ is the so-called carr\'{e} du champ operator. We will need to
impose the following uniform upper and lower bounds on $a.$

\begin{definition}
For $\Lambda\geq1$ we define $\Xi(\Lambda)$ to be the class of such measurable
maps $a:\mathfrak{g}$ $\rightarrow\mathcal{S}_{d}$ having the property that
\begin{equation}
\forall y\in\mathbb{R}^{d}:\Lambda^{-1}|y|^{2}\leq\sup_{x\in\mathfrak{g}}%
y^{T}a(x)y\leq\Lambda|y|^{2}. \label{ellipticity}%
\end{equation}

\end{definition}

\begin{remark}
When $a\left(  \cdot\right)  \equiv I_{d}$, the identity matrix, we write
$\mathcal{E},\Gamma$ etc. in place of $\mathcal{E}^{a},\Gamma^{a}.$
\end{remark}

\qquad We can introduce the operator $\mathcal{L}^{a}:C_{c}^{\infty}\left(
\mathfrak{g}\right)  \rightarrow L^{2}\left(  m\right)  $ by
\[
\mathcal{L}^{a}f:=\sum_{i,j=1}^{d}B_{j}\left[  a^{ij}(\cdot)B_{i}%
f(\cdot)\right]  \left(  h\right)  ,
\]
where the right-hand-side is defined in $L^{2}\left(  m\right)  $ by weak
integration. Integrating-by-parts in (\ref{Dirichlet_form}), and using the
left-invariance of the $B_{j}$ vector fields together with the
right-invariance of $m$ ensures that
\[
\mathcal{E}^{a}(f,g)=-\left\langle \mathcal{L}^{a}f,g\right\rangle
_{L^{2}\left(  m\right)  }.
\]
The operator $\mathcal{L}^{a}:C_{c}^{\infty}\left(  \mathfrak{g}\right)
\rightarrow L^{2}\left(  m\right)  $ is then symmetric, but not self-adjoint
since $\mathcal{E}^{a}\left(  f\right)  :=\mathcal{E}^{a}\left(  f,f\right)  $
is defined a priori only on $C_{c}^{\infty}\left(  \mathfrak{g}\right)
\subset L^{2}\left(  m\right)  $. Nevertheless, the (non-negative) bilinear
form $\mathcal{E}^{a}\left(  f,g\right)  $ has the property of being closable,
and is also strongly local and strongly regular in the sense of
{\cite[Corollary 4.2]{Sturm1996a}} (see also \cite{FrizVictoir2010}). \ It is
therefore possible to construct a self-adjoint extension using the classical
Friedrich's procedure (see for instance \cite{Kato1995}). The idea is first to
find a suitable domain on which to extend $\mathcal{E}^{a}$, and then out of
this constructing a domain for extending $\mathcal{L}^{a}$.

\qquad In a little more detail, we first take the completion of $C_{c}%
^{\infty}\left(  \mathfrak{g}\right)  $ with respect to the norm%
\[
\left\vert \left\vert f\right\vert \right\vert _{\mathcal{E}^{a}}:=\left(
\left\vert \left\vert f\right\vert \right\vert _{L^{2}\left(  m\right)  }%
^{2}+\mathcal{E}^{a}\left(  f\right)  \right)  ^{1/2}.
\]
This completion may then be identified with a naturally embedded Hilbert
subspace of $L^{2}\left(  m\right)  $ -- the Dirichlet domain -- which we
denote $D\left(  \mathcal{E}^{a}\right)  .$ This then allows for the
construction of a dense linear subspace $D\left(  \mathcal{L}^{a}\right)  $ of
$L^{2}\left(  m\right)  $ on which there exists a self-adjoint extension of
$\mathcal{L}^{a}$, which is customarily also denoted $\mathcal{L}^{a}$. \ With
the self-adjoint extension in hand, classical theory provides the route from
$\mathcal{L}^{a}$ to a semi-group of contractions on $L^{2}\left(  m\right)
$, $\ $which we will denote $\left(  P_{t}^{a}\right)  _{t\geq0},$ and thence
to an associated Markov process. \qquad It is often insightful to make
comparisons of the probabilistic features of different Markov processes
produced by this procedure by comparing their Dirichlet forms. To this end, we
recall the intrinsic distance associated with $\mathcal{E}^{a}$, which is
defined by
\begin{equation}
d^{a}(x,y)=\sup\left\{  f(x)-f(y):f\in\mathcal{F}_{\text{loc}}\text{ and
}f\text{ continuous},\Gamma^{a}(f,f)\leq1\right\}  , \label{in dist}%
\end{equation}
where $\mathcal{F}_{\text{loc}}:=\left\{  f\in L^{2}\left(  dm\right)
:\Gamma^{a}(f,f)\in L_{\text{loc}}^{1}(dm)\right\}  .$ The following lemma is
proved in \cite{FrizVictoir2010} based partly on results in {\cite[p.285]%
{Stroock1988}}

\begin{lemma}
\label{compare}Let $\Lambda\geq1$, then $D(\mathcal{E}^{a})=D(\mathcal{E})$
for every $a\in\Xi(\Lambda)$ and
\[
\frac{1}{\Lambda}\mathcal{E}(f)\leq\mathcal{E}^{a}(f)\leq\Lambda
\mathcal{E}(f),\text{ for all }f\in D(\mathcal{E}).
\]
The intrinsic distance $d$ associated with $\mathcal{E}$ coincides with the
Carnot Carath\'{e}odory metric on $\mathfrak{g,}$ and furthermore $d$ and
$d^{a}$ are Lipschitz equivalent
\end{lemma}

This semi-group $\left(  P_{t}^{a}\right)  _{t\geq0}$ referred to above is
easily seen by Sobolev estimates (see, e.g., \cite{Davies1989}) to admit a
kernel representation of the form
\[
(P_{t}^{a}f)\left(  x\right)  =\int f(y)p^{a}(t,x,y)\,dy.
\]
The heat kernel $p^{a}$ can be shown to satisfy the following Aronson-type
estimate (see {\cite[Corollary 4.2]{Sturm1996a}}).

\begin{theorem}
Let $a\in\Xi(\Lambda)$. The heat kernel $p^{a}$ associated with the Dirichlet
form $\mathcal{E}^{a}$ satisfies, for $\epsilon>0$ fixed,
\begin{equation}
p^{a}(t,x,y)\leq\frac{C}{\sqrt{t^{\dim_{H}\mathfrak{g}}}}\exp\left(
-\frac{d^{a}(x,y)^{2}}{(4+\epsilon)t}\right)  .
\end{equation}
for some constant $C=C(\epsilon,\Lambda)$. Here $\dim_{H}\mathfrak{g}$ denotes
the so called homogeneous dimension of $\mathfrak{g}$.
\end{theorem}

Lemma \ref{compare} allows us to compare these transition densities for
different $a$ very effectively. In particular, it is immediate from the
Lipschitz equivalence of $d$ and $d^{a}$ that we have the following result.

\begin{corollary}
\label{heatkernel_bound} Let $a\in\Xi(\Lambda)$ and $\epsilon>0$ fixed, then
the heat kernel $p^{a}$ satisfies
\[
p^{a}(t,x,y)\leq\frac{C}{\sqrt{t^{\dim_{H}\mathfrak{g}}}}\exp\left(
-\frac{d(x,y)^{2}}{(4+\epsilon)\Lambda t}\right)  .
\]
with constant $C=C(\epsilon,\Lambda)$.
\end{corollary}

The heat kernel $p^{a}$ allows for a consistent family of finite-dimensional
distributions and thus determines a $\mathfrak{g}$-valued (strong) Markov
process $(\mathbf{X}_{t}^{a,x})_{t\geq0}$ with $a\in\Xi(\Lambda)$ and
$\mathbf{X}_{0}^{a,x}=x\in\mathfrak{g}$. Using Kolmogorov's criterion (see
\cite[Theorem 13]{FrizVictoir2008}) it can be shown that $\mathbf{X}^{a,x}$
has a version with continuous sample paths, i.e. $\mathbf{X}^{a,x}\in
C^{1/p-\text{H\"{o}l}}([0,\infty),\mathfrak{g})$ . In fact, much more can be
shown; the following theorem is an assembly of results from {\cite[Theorem
13]{FrizVictoir2008}} which we will need subsequently.

\begin{theorem}
\label{Holder_bound} Let $\Lambda\geq1$ and suppose $N\geq2$ is a natural
number. Assume $a\in\Xi(\Lambda)$ and $x\in\mathfrak{g=g}^{N}$. For any $p>2$
there exists a version $\mathbf{X}^{a,x}$ which takes values in
$C^{1/p-\text{H\"{o}l}}([0,\infty),\mathfrak{g}).$ In particular when
$p\in\lbrack N,N+1)$ we have that $\mathbf{X}^{a,x}$ restricted to $\left[
0,T\right]  $ is in $WG\Omega_{p}\left(  \left[  0,T\right]  ,V\right)  .$
Letting $\mathbb{P}^{a,x}$ denote the probability measure on $C([0,\infty
),\mathfrak{g})$ given by the law of $\mathbf{X}^{a,x}$. Then there exists a
finite constant $C=C(\alpha,\Lambda,T, N)$ such that
\[
\sup_{x\in\mathfrak{g}}\mathbb{P}^{a,x}\left(  \sup_{[s,t]\subseteq
\lbrack0,T]}\frac{d^{a}\left(  \mathbf{X}_{s},\mathbf{X}_{t}\right)
}{|t-s|^{\alpha}}>r\right)  \leq C\exp\left(  -\frac{r^{2}}{C}\right)  ,
\]
wherein $\mathbf{X}_{s}\left(  \mathfrak{\omega}\right)  \mathfrak{=\omega
}\left(  s\right)  $ for $s\geq0$ denotes the evaluation maps on
$C([0,\infty),\mathfrak{g}).$ Moreover $\mathbf{X}^{a,x}$ satisfies the
following weak scaling property for all $r>0$
\begin{equation}
\left(  \mathbf{X}_{t}^{a^{r},x}:t\geq0\right)  \overset{\mathcal{D}%
}{=}\left(  \delta_{r}\mathbf{X}_{tr^{-2}}^{a,\delta_{r^{-1}}(x)}%
:t\geq0\right)  , \label{scaling}%
\end{equation}
where $\delta_{r}$ denotes the natural scaling operation on $\mathfrak{g,}$
$\overset{\mathcal{D}}{=}$ denotes equality in distribution and $a^{r}%
(x):=a\left(  \delta_{r^{-1}}x\right)  \in\Xi(\Lambda)$.
\end{theorem}

\section{A large deviations result\label{ld}}

Let $d_{CC}$ be the Carnot-Carath\'{e}odory metric on $\mathfrak{g}$ that was
introduced in Section \ref{rough path}. Given any $\mathbf{x}$ in $C\left(
[0,\infty),\mathfrak{g}\right)  $ and $r>0$ we can define inductively a
non-decreasing sequence $\left(  \sigma_{n}^{r}\right)  _{n=0}^{\infty
}=\left(  \sigma_{n}^{r}\left(  \mathbf{x}\right)  \right)  _{n=0}^{\infty}$
by setting $\sigma_{0}^{r}=0,$ and then for $n\in%
%TCIMACRO{\U{2115} }%
%BeginExpansion
\mathbb{N}
%EndExpansion
$
\begin{equation}
\sigma_{n}^{r}:=\inf\left\{  t\geq\sigma_{n-1}^{r}:d_{CC}\left(
\mathbf{x}_{\sigma_{n-1}},\mathbf{x}_{t}\right)  \geq r\right\}  . \label{st}%
\end{equation}

\begin{definition}
\label{stopping_times} For any $T\geq0$ we define the functional $N_{0}%
^{r}(\mathbf{\cdot})=N_{0}^{r}\left(  \mathbf{\cdot},[0,T]\right)  :C\left(
[0,\infty),\mathfrak{g}\right)  \rightarrow%
%TCIMACRO{\U{2115} }%
%BeginExpansion
\mathbb{N}
%EndExpansion
\cup\left\{  0\right\}  $ by
\[
N_{0}^{r}(\mathbf{x,}[0,T]):=\sup\left\{  n:\sigma_{n}^{r}<T\right\}  .
\]

\end{definition}

\begin{remark}
When $r=1$ we will omit the superscripts and write $\sigma_{n},N_{0}(\mathbf{\cdot})$ and so forth. Note that $N_{0}^{r}(\mathbf{x,}
[0,T])<\infty$ implies that the set
\[
\left\{  \sigma_{j}:j=0,1,...,N_{0}^{r}\left(  \mathbf{x,}[0,T]\right)
\right\}  \cup\left\{  T\right\}
\]
forms a partition of the interval $[0,T]$.
\end{remark}

It is our goal in this section to analyse the tail behaviour of the integer
valued random variables $N_{0}^{r}(\mathbf{X}^{a,x},\left[  0,T\right]  )$,
when $\mathbf{X}^{a,x}$ is a Markovian rough path of the type described in
Section\ \ref{markov}. Our approach will be motivated by the following
well-known example.

\begin{example}
[Brownian motion]Let $B=(B_{t})_{t\geq0}$ a one-dimensional standard Brownian
motion on some probability space $\left(  \Omega,\mathcal{F},\mathbb{P}%
\right)  $. In this setting, the sequence in (\ref{st}) is given by
\[
\sigma_{0}:=0,\quad\sigma_{n+1}:=\inf\left\{  t\geq\sigma_{n}:|B_{t}%
-B_{\sigma_{n}}|\geq1\right\}  .
\]
It is a classical result (see, e.g., \cite{KaratzasShreve1991}) that the
Laplace transform of $\sigma:=\sigma_{1}$ satisfies
\begin{equation}
\mathbb{E}\left[  e^{-\lambda\sigma}\right]  =\cosh\left(  \sqrt{2\lambda
}\right)  ^{-1}\leq2e^{-\sqrt{2\lambda}}. \label{closed form}%
\end{equation}
If we let $\xi_{k}:=\sigma_{k}-\sigma_{k-1}$ for $k=1,...,n$ and note that
$\left\{  \xi_{k}:k=1,...,n\right\}  $ are i.i.d. with each $\xi_{k}%
\,${\ equal in distribution to }$\,\sigma$, then using $\sum_{k=1}^{n}\xi
_{k}=\sigma_{n}$, it follows that for all $\theta>0$
\[
\mathbb{P}\left(  N_{0}(B,\left[  0,1\right]  )\geq n\right)  =\mathbb{P}%
\left(  \sigma_{n}<1\right)  \leq e^{\theta}\mathbb{E}\left[  e^{-\theta
\sigma}\right]  ^{n}\leq2^{n}e^{\theta}e^{-n\sqrt{2\theta}}.
\]
The last expression can be minimized by the choice $\theta=2^{-1}n^{2}$, which
immediately yields the estimate $\mathbb{P}(N_{0}(B,\left[  0,1\right]  )\geq
n)\leq2^{n}e^{-\frac{n^{2}}{2}}\leq c_{1}e^{-c_{2}n^{2}},$ for some $c_{1}$
and $c_{2}$ in $\left(  0,\infty\right)  $ which do not depend on $n.$
\end{example}

This example makes clear the importance of the Laplace transform when
analysing the tail behaviour of $N_{0}^{r}(\mathbf{X}^{a,x}\mathbf{,}\left[
0,T\right]  )$. What is important, as we will show, is not to have a
closed-form expression as in (\ref{closed form}), but instead to have an upper
bound controlling its asymptotic behaviour as $\lambda\rightarrow\infty.$

\subsection{\bigskip Tails for $N_{0}^{r}\left(  \mathbf{X}^{a,x}%
\mathbf{,}[0,T]\right)  $}

From now on we fix $\Lambda\geq1$, $N\geq2$ and let $\mathfrak{g=g}^{N}$. We
will adopt the notation of Theorem \ref{Holder_bound}, i.e. for $a\in
\Xi\left(  \Lambda\right)  $ and $x\in\mathfrak{g,}$ $\mathbf{X}^{a,x}\in
C([0,\infty),\mathfrak{g})$ will be the strong Markov process associated with
$\mathcal{E}^{a},$ $\mathbb{P}^{a,x}$ will be the law of on $C\left(
[0,\infty),\mathfrak{g}\right)  $ and $\mathbb{E}^{a,x}$ the corresponding
expectation. For $t>0$ we continue to denote the evaluation maps by%
\[
\mathbf{X}_{t}:C([0,\infty),\mathfrak{g})\rightarrow C([0,\infty
),\mathfrak{g}),\text{ }\mathbf{X}_{t}\left(  \mathfrak{\omega}\right)
\mathfrak{=\omega}\left(  t\right)  \text{ for }t\geq0.
\]
We introduce the following notation for the Laplace transform of
$\sigma=:\sigma_{1}\left(  \mathbf{X}\right)  $, the random variable in
(\ref{st}) with $r=1,$ under the probability measure $\mathbb{P}^{a,x}:$
\begin{equation}
M(\lambda;a,x):=\mathbb{E}^{a,x}\left[  e^{-\lambda\sigma}\right]
=\int_{C\left(  [0,\infty),\mathfrak{g}\right)  }e^{-\lambda\sigma\left(
\omega\right)  }\mathbb{P}^{a,x}\left(  d\omega\right)  , \label{laplace}%
\end{equation}
ready for stating a version of De Bruijn's exponential Tauberian theorem. This
result, in a precise way, relates the asymptotic behaviour of the log Laplace
transform, $\log M(\lambda;a,x)$ as $\lambda\rightarrow\infty,$ and the log
short-time probability $\log\mathbb{P}^{a,x}\left(  \sigma\leq t\right)  $ as
$t\rightarrow0+.$

\begin{lemma}
[Exponential Tauberian theorem]Let $c>0$. The following two statements are equivalent:

\begin{enumerate}
\item $-\log M(\lambda;a,x)\sim c\sqrt{\lambda},\text{as }\lambda
\rightarrow\infty$;

\item $-\log\mathbb{P}^{a,x}\left(  \sigma\leq t\right)  \sim\frac{c^{2}}%
{4t},\text{ as }t\rightarrow0+.$
\end{enumerate}
\end{lemma}

\begin{proof}
This is an immediate consequence of applying Theorem 4.12.9 in
\cite{BinghamGoldieTeugels1987}, \ making the choice $B=\frac{c^{2}}{4}$ and
$\phi(\lambda)=\frac{1}{\lambda}$ in the notation of that theorem.
\end{proof}

\qquad We will not need the full strength of this equivalence. Instead, we
will need the following statement which relates the asymptotic oscillations of
the two functions. We give a short proof for completeness and refer the reader
to {\cite{BinghamGoldieTeugels1987}} for further discussion of results of this type.

\begin{lemma}
\label{Tauberian} Assume that $\Lambda\geq1,$ $a\in\Xi(\Lambda)$ and
$x\in\mathfrak{g,}$ and let $\mathbb{P}^{a,x}$ denote the law of the rough
path $\mathbf{X}^{a,x}$ on $C\left(  [0,\infty),\mathfrak{g}\right)  .$
Let\ $\sigma=\sigma_{1}$ be the stopping time defined in (\ref{st}) with
$r=1$, and suppose there exists $c>0$ for which
\begin{equation}
\limsup_{t\rightarrow0+}t\sup_{a\in\Xi(\Lambda)}\sup_{x\in\mathfrak{g}}%
\log\mathbb{P}^{a,x}(\sigma\leq t)\leq-c. \label{assumption}%
\end{equation}
Then
\[
\limsup_{\lambda\rightarrow\infty}{\lambda}^{-\frac{1}{2}}\sup_{a\in
\Xi(\Lambda)}\sup_{x\in\mathfrak{g}}\log M(\lambda;a,x)\leq-2\sqrt{c}.
\]

\end{lemma}

\begin{proof}
Set $\mathbb{P}^{a,x}(\sigma\leq t)=:\mu_{a,x}(t)$. And note that it is
sufficient to show that the assumption (\ref{assumption}) implies
\[
\limsup_{\lambda\rightarrow0+}\lambda\sup_{a\in\Xi(\lambda)}\sup
_{x\in\mathfrak{g}}\log M\left(  \frac{1}{\lambda^{2}};a,x\right)  \leq
-2\sqrt{c}.
\]
Observe that for $\xi,\lambda>0$,
\begin{align}
M\left(  \frac{1}{\lambda^{2}};a,x\right)   &  =\int_{0}^{\frac{\lambda}{\xi}%
}\exp\left(  -\frac{t}{\lambda^{2}}\right)  \,d\mu_{a,x}(t)+\int%
_{\frac{\lambda}{\xi}}^{\infty}\exp\left(  -\frac{t}{\lambda^{2}}\right)
\,d\mu_{a,x}(t)\nonumber\\
&  \leq\mu_{a,x}\left(  \frac{\lambda}{\xi}\right)  +\exp\left(  -\frac{1}%
{\xi\lambda}\right)  . \label{M bound}%
\end{align}
Next we use this bound and exploit the well-know fact that for any sequences
$\left(  a_{n}\right)  _{n=1}^{\infty}$ and $\left(  b_{n}\right)
_{n=1}^{\infty}$ of positive real numbers we have%
\[
\limsup_{n\rightarrow\infty}\frac{1}{n}\log\left(  a_{n}+b_{n}\right)
\leq\limsup_{n\rightarrow\infty}\frac{1}{n}\log a_{n}+\limsup_{n\rightarrow
\infty}\frac{1}{n}\log b_{n}.
\]
In the setting of (\ref{M bound}) this gives
\begin{align*}
\limsup_{\lambda\rightarrow0+}\lambda\sup_{a,x}\log M\left(  \frac{1}%
{\lambda^{2}};a,x\right)   &  \leq\limsup_{\lambda\rightarrow0+}\lambda
\sup_{a,x}\log\left(  \mu_{a,x}\left(  \frac{\lambda}{\xi}\right)  \right)
-\frac{1}{\xi}\\
&  \leq-\xi c-\frac{1}{\xi},
\end{align*}
where the last line uses the hypothesis (\ref{assumption}). Because the
function $(0,\infty)\ni\xi\mapsto-(\xi c+\frac{1}{\xi})$ attains its global
maximum $-2\sqrt{c}$ at $\xi^{\star}=c^{-\frac{1}{2}}$, we obtain
\[
\limsup_{\lambda\rightarrow0+}\lambda\sup_{a\in\Xi(\lambda)}\sup
_{x\in\mathfrak{g}}\log M\left(  \frac{1}{\lambda^{2}};a,x\right)  \leq
-2\sqrt{c}%
\]
which completes the proof.
\end{proof}

The following lemma will make the previous result applicable to our setting.

\begin{lemma}
\label{BoundedProba} Denote by $\mathbb{P}^{a,x}$ and $\sigma$, respectively,
the probability measure and stopping time defined in the statement of lemma
\ref{Tauberian}. There exist constants $c_{1},c_{2}\in(0,\infty)$, which
depend only on $d$, $\Lambda$ and $N,$ such that for all $t\in(0,1]$
\[
\mathbb{P}^{a,x}\left(  \sigma\leq t\right)  \leq c_{1}\exp\left(
-\frac{c_{2}}{t}\right)  .
\]

\end{lemma}

As a consequence of this lemma we see that
\[
\limsup_{t\rightarrow0+}t\sup_{a\in\Xi(\Lambda)}\sup_{x\in\mathfrak{g}}%
\log\mathbb{P}^{a,x}\left(  \sigma\leq t\right)  =-c_{2}<0,
\]
which allows us to apply lemma \ref{Tauberian} and immediately deduce the
following corollary.

\begin{corollary}
\label{Tauberian_Remark}Let $M(\lambda;a,x)$ denote the Laplace transform
(\ref{laplace}) of the stopping time $\sigma.$ Then under the condition of
lemma \ref{BoundedProba} we have
\[
\limsup_{\lambda\rightarrow\infty}\sqrt{\lambda}^{-1}\sup_{a\in\Xi(\Lambda
)}\sup_{x\in\mathfrak{g}}\log M(\lambda;a,x)\leq-2\sqrt{c_{2}},
\]
and hence there exists a constant $\lambda_{0}\in\left(  0,\infty\right)  $
such that
\begin{equation}
\sup_{a\in\Xi(\Lambda)}\sup_{x\in\mathfrak{g}}M(\lambda;a,x)\leq\exp\left(
-\sqrt{c_{2}\lambda}\right)  \text{ for all }\lambda\geq\lambda_{0}.
\label{unif}%
\end{equation}

\end{corollary}

\begin{proof}
[Proof of Lemma \ref{BoundedProba}]We follow {\cite[Proposition 6.5]%
{Bass1998}} where a similar upper bound is obtained in the case of uniformly
elliptic diffusions. By using the Gaussian upper estimate in Corollary
\ref{heatkernel_bound} we will adapt the proof for the class of Markovian
rough paths introduced earlier. First we note that%
\[
\mathbb{P}^{a,x}\left(  \sigma\leq t\right)  \leq\mathbb{P}^{a,x}\left(
\sigma\leq t,d_{CC}\left(  \mathbf{X}_{t},x\right)  <\frac{1}{2}\right)
+\mathbb{P}^{a,x}\left(  d_{CC}\left(  \mathbf{X}_{t},x\right)  \geq\frac
{1}{2}\right)  .
\]
For $g\in\mathfrak{g}$ and $\delta>0$ we let $B\left(  g,\delta\right)  $
denote the open $d_{CC}$-ball of radius $\delta$ centred at\ $g.$ Then using
Corollary \ref{heatkernel_bound} (with fixed $\epsilon>0$), we see that the
second term satisfies
\begin{equation}%
\begin{split}
\mathbb{P}^{a,x}\left(  d_{CC}\left(  \mathbf{X}_{t},x\right)  \geq\frac{1}%
{2}\right)   &  =\int_{B\left(  x,\frac{1}{2}\right)  ^{c}}p^{a}(t,x,y)\,dy\\
&  \leq\int_{\frac{1}{2}}^{\infty}\frac{c_{1}}{\sqrt{t^{\dim_{H}\mathfrak{g}}%
}}\exp\left(  -c_{2}\frac{r^{2}}{t}\right)  r^{\dim_{H}\mathfrak{g}-1}\,dr\\
&  =\int_{\frac{1}{2\sqrt{t}}}^{\infty}c_{1}v^{\dim_{H}\mathfrak{g}-1}%
\exp\left(  -c_{2}v^{2}\right)  \,dv\\
&  \leq c_{3}e^{-\frac{c_{4}}{t}},
\end{split}
\label{first_bound}%
\end{equation}
where the constants $c_{3}$ and $c_{4}$ depend only on $d$, $\Lambda$ and $N.$
For the first term, observe that
\[%
\begin{split}
\mathbb{P}^{a,x}\left(  \sigma_{1}\leq t,d_{CC}\left(  \mathbf{X}%
_{t},x\right)  <\frac{1}{2}\right)   &  \leq\int_{0}^{t}\mathbb{P}%
^{a,x}\left(  \sigma_{1}\in ds,d_{CC}\left(  \mathbf{X}_{t},\mathbf{X}%
_{\sigma_{1}}\right)  \geq\frac{1}{2}\right) \\
&  =\int_{0}^{t}\mathbb{E}^{a,x}\left[  1_{\left\{  \sigma_1\in ds\right\}
}\mathbb{P}^{a,\mathbf{X}_{\sigma_{1}}}\left(  d_{CC}\left(  \mathbf{X}%
_{t-\sigma_{1}},\mathbf{X}_{0}\right)  \geq\frac{1}{2}\right)  \right] \\
&  \leq\int_{0}^{t}\mathbb{E}^{a,x}\left[  1_{\left\{  \sigma_1\in ds\right\}
}\mathbb{P}^{a,\mathbf{X}_{s}}\left(  d_{CC}\left(  \mathbf{X}_{t-s}%
,\mathbf{X}_{0}\right)  \geq\frac{1}{2} \right) \right]  .
\end{split}
\]
By the same argument as in (\ref{first_bound}), we know there exist constants
$c_{5}$ and $c_{6}$ (which, again, depend only on $d$, $\Lambda$ and $N)$ such
that
\[
\sup_{r\leq t}\mathbb{P}^{a,x}\left(  d_{CC}\left(  \mathbf{X}_{r}%
,\mathbf{X}_{0}\right)  \geq\frac{1}{2}\right)  \leq c_{5}e^{-\frac{c_{6}}{t}%
}.
\]
Together these bounds imply the desired result. The statement (\ref{unif})
then follows in a straight forward way.
\end{proof}

We can now prove the needed tail estimates for the the random variables
$N_{0}^{r}\left(  \mathbf{X,}\left[  0,T\right]  \right)  $ under
$\mathbb{P}^{a,x}.$

\begin{remark}
\label{scale remark}We will use the fact that the distrubution of $\sigma
^{r}=\sigma_{1}^{r}$ under $\mathbb{P}^{a,x}$ equals the distribution of
$r^{2}\sigma=r^{2}\sigma_{1}^{1}$ under $\mathbb{P}^{a^{1/r},\delta_{1/r}x}.$
This is a consequnce of the scaling property (\ref{scaling}).
\end{remark}

\begin{proposition}
\label{Gaussian_Tail}Let $\Lambda\geq1,$ $a\in\Xi(\Lambda).$ Assume that
$\mathbf{X}^{a,x}$ is the $\mathfrak{g-}$valued Markov process, defined on
some probability space, associated with the Dirichlet form
(\ref{Dirichlet_form}). Let $\mathbb{P}^{a,x}$ be the (Borel) probability
measure on $C\left(  [0,\infty),\mathfrak{g}\right)  $ and $c_{2},\lambda
_{0}\in\left(  0,\infty\right)  $ be the constants in (\ref{unif}). For every
$r>0$ the random variable $N_{0}^{r}\left(  \cdot\mathbf{,}\left[  0,T\right]
\right)  :$ $C\left(  [0,\infty),\mathfrak{g}\right)  \rightarrow%
%TCIMACRO{\U{2115} }%
%BeginExpansion
\mathbb{N}
%EndExpansion
\cup\left\{  0\right\}  $ in Definition \ref{stopping_times} satifies
\begin{equation}
\mathbb{P}^{a,x}\left(  N_{0}^{r}\left(  \mathbf{X,}\left[  0,T\right]
\right)  \geq n\right)  \leq\exp\left(  -\frac{c_{2}n^{2}r^{2}}{4T}\right)
\label{tail1}%
\end{equation}
for all $n\geq2T\lambda_{0}^{1/2}r^{-2}c_{2}^{-1/2}.$
\end{proposition}

\begin{proof}
As previously we write $\sigma_{n}^{r}=\sum_{k=1}^{n}\xi_{k}^{r}$, where
$\xi_{k}^{r}=\sigma_{k}^{r}-\sigma_{k-1}^{r}$. We aim to estimate the
probability in (\ref{tail1}), to do so we note that for $\lambda>0$ we have%
\begin{equation}
\mathbb{P}^{a,x}\left(  N_{0}^{r}\left(  \mathbf{X,}\left[  0,T\right]
\right)  \geq n\right)  \leq e^{\lambda T}\mathbb{E}^{a,x}\left[
e^{-\lambda\sum_{k=1}^{n}\xi_{k}^{r}}\right]  . \label{cheby}%
\end{equation}
By using the scaling property in the manner of Remark \ref{scale remark} gives
that%
\[
M_{r}(\lambda;a,x):=\mathbb{E}^{a,x}\left[  e^{-\lambda\xi_{1}^{r}}\right]
=\mathbb{E}^{a,x}\left[  e^{-\lambda\sigma_{1}^{r}}\right]  =\mathbb{E}%
^{^{a^{1/r},\delta_{1/r}x}}\left[  e^{-\lambda r^{2}\sigma}\right]  =M(\lambda
r^{2};a^{1/r},\delta_{1/r}x),
\]
whereupon the inequality (\ref{unif}) in Corollary yields
\begin{equation}
\sup_{a\in\Xi(\Lambda)}\sup_{x\in\mathfrak{g}}M_{r}(\lambda;a,x)\leq
\exp\left(  -\sqrt{c_{2}\lambda}r\right)  \text{ for all }\lambda\geq
\lambda_{0}r^{-2}. \label{rmk}%
\end{equation}
Combining the Strong Markov Property at the stopping time $\sigma_{n-1}^{r}$
with an easy induction yields the estimate%
\begin{align}
\mathbb{E}^{a,x}\left[  e^{-\lambda\sum_{k=1}^{n}\xi_{k}^{r}}\right]   &
=\mathbb{E}^{a,x}\left[  e^{-\lambda\sum_{k=1}^{n-1}\xi_{k}^{r}}%
\mathbb{E}^{a,\mathbf{X}_{\sigma_{n-1}^{r}}}\left[  e^{-\lambda\sigma_{1}^{r}%
}\right]  \right] \nonumber\\
&  \leq\mathbb{E}^{a,x}\left[  e^{-\lambda\sum_{k=1}^{n-1}\xi_{k}^{r}}\right]
\sup_{a\in\Xi(\Lambda)}\sup_{x\in\mathfrak{g}}M_{r}(\lambda
;a,x)\label{Strong_Markov_Property}\\
&  \leq\sup_{a\in\Xi(\Lambda)}\sup_{x\in\mathfrak{g}}M_{r}(\lambda
;a,x)^{n}\nonumber
\end{align}
Using with (\ref{Strong_Markov_Property}) and\ (\ref{cheby}) gives, for all
$\lambda\geq\lambda_{0}r^{-2},$
\[
\mathbb{P}^{a,x}\left(  N_{0}^{r}\left(  \mathbf{X,}\left[  0,T\right]
\right)  \geq n\right)  \leq\exp\left(  \lambda T-n\sqrt{c_{2}\lambda
}r\right)  .
\]
Because the right hand side is minimized by the choice $\lambda=T^{-2}%
4^{-1}c_{2}n^{2}r^{2}$, we have that
\[
\mathbb{P}^{a,x}\left(  N_{0}(\mathbf{X,}\left[  0,T\right]  )\geq n\right)
\leq\exp\left(  -\frac{c_{2}n^{2}r^{2}}{4T}\right)  ,
\]
provided $T^{-2}4^{-1}c_{2}n^{2}r^{2}\geq\lambda_{0}r^{-2},$ i.e. if
$n\geq2T\lambda_{0}^{1/2}r^{-2}c_{2}^{-1/2}.$
\end{proof}

\section{Tail estimates for the accumulated local $p$-variation\label{main}}

The law of the sub-elliptic Markov process $\mathbb{P}^{a,x}$ constructed in
Section \ref{markov} is, for any $p>2$ and $T\geq0,$ supported in
$C^{1/p-\text{H\"{o}l}}([0,T],\mathfrak{g})\subset C^{p-\text{var}%
}([0,T],\mathfrak{g})\subset C\left(  [0,T],\mathfrak{g}\right)  .$ This
observation allows us to go beyond the analysis of the previous section and
address the tail behaviour of the\textit{\ accumulated local }$p$%
\textit{-variation.} We first recall the definition of this functional (cf.
\cite{CLL2013})

\begin{definition}
[accumulated local $p$-variation]\label{accumul}Let $p\geq1$. We define the
accumulated local $p$-variation to be the function $M\left(  \mathbf{\cdot
},[0,T]\right)  =M\left(  \mathbf{\cdot}\right)  :C^{p-\text{var}}%
([0,\infty),\mathfrak{g})\rightarrow%
%TCIMACRO{\U{211d} }%
%BeginExpansion
\mathbb{R}
%EndExpansion
_{\geq0}$ by
\[
M(\mathbf{x},[0,T]):=\sup_{\overset{D=(t_{i})}{\omega_{\mathbf{x}}%
(t_{i},t_{i+1})\leq1}}\sum_{i}\omega_{\mathbf{x}}\left(  t_{i},t_{i+1}\right)
,
\]
where $\omega_{\mathbf{x}}\left(  s,t\right)  \equiv\left\vert \left\vert
\mathbf{x}\right\vert \right\vert _{p\text{-var;}\left[  s,t\right]  }^{p}$ is
the control induced by $\mathbf{x,}$ and the supremum is taken over all
partitions $D$ of the interval $[0,T]$ such that $\omega_{\mathbf{x}},$ when
evaluated between two consecutive points in $D,$ is bounded by one.
\end{definition}

We will now show that the accumulated local $p$-variation of $\mathbf{x}$ over
$\left[  0,T\right]  $ can be bounded by the sum of $N_{0}(\mathbf{x,}\left[
0,T\right]  )$ and the accumulated $p$-variation between the times associated
with the sequence $\sigma_{i}$, $i=0,1,...,N_{0}(\mathbf{x,}\left[
0,T\right]  )$.

\begin{lemma}
\label{key est}Let $p\geq1$, assume $\mathbf{x}\in C^{p-\text{var}}%
([0,\infty),\mathfrak{g})$ and suppose $\omega_{\mathbf{x}}$ is the control
induced by $\mathbf{x.}$ Let $r>0$ and write $\left(  \sigma_{n}^{r}\left(
\mathbf{x}\right)  \right)  _{n=0}^{\infty}=\left(  \sigma_{n}^{r}\right)
_{n=0}^{\infty}$ for the sequence and $N_{0}^{r}(\mathbf{x,}\left[
0,T\right]  )$ for the non-negative integer defined by (\ref{st}) and in
Definition \ref{stopping_times}, respectively. Then we can bound
$M(\mathbf{x},[0,T])$, the accumulated local $p-$variation, using the
following estimate%
\begin{equation}
M(\mathbf{x},[0,T])\leq N_{0}^{r}(\mathbf{x,}\left[  0,T\right]  )+\sum
_{j=1}^{N_{0}^{r}(\mathbf{x,}\left[  0,T\right]  )}\omega_{\mathbf{x}}\left(
\sigma_{j-1}^{r},\sigma_{j}^{r}\right)  +\omega_{\mathbf{x}}\left(
\sigma_{N_{0}^{r}(\mathbf{x,}\left[  0,T\right]  )}^{r},T\right)  .
\label{key est1}%
\end{equation}

\end{lemma}

\begin{proof}
To deal with the end point in a notationally efficient way, we redefine
$\sigma_{N_{0}^{r}(\mathbf{x,}\left[  0,T\right]  )+1}$ so that $\sigma
_{N_{0}^{r}(\mathbf{x,}\left[  0,T\right]  )+1}=T$ (Note: this is, strictly,
an abuse of notation in light of the definition in formula (\ref{st})).
Suppose then that $D=\left\{  0=t_{0}<t_{1}<...<t_{n}=T\right\}  $ is an
arbitrary partition of $\left[  0,T\right]  ,$ such that any two consecutive
points $s<t$ in $D$ satisfy $\omega_{\mathbf{x}}(s,t)\leq1$. We define the
function $\Phi:\left\{  0,1,...,n-1\right\}  \rightarrow\left\{
0,1,...,N_{0}^{r}(\mathbf{x,}\left[  0,T\right]  )\right\}  $ by
\[
\Phi\left(  i\right)  =\max\left\{  k\in%
%TCIMACRO{\U{2115} }%
%BeginExpansion
\mathbb{N}
%EndExpansion
\cup\left\{  0\right\}  :\sigma_{k}^{r}\leq t_{i}\right\}  \text{ for
}i=0,1...,n-1,
\]
and then let $A$ denote the subset
\[
A=\left\{  k<N_{0}^{r}(\mathbf{x,}\left[  0,T\right]  ):\exists i\text{ with
}\Phi\left(  i\right)  =k\right\}  \subseteq\left\{  0,1,...,N_{0}%
^{r}(\mathbf{x,}\left[  0,T\right]  )-1\right\}  .
\]
For each $k\in A$ we define
\[
m_{k}=\min\left\{  i:\Phi\left(  i\right)  =k\right\}  \text{ and }n_{k}%
=\max\left\{  i:\Phi\left(  i\right)  =k\right\}  .
\]

whereupon it is an easy consequence of the definitions of $m_{k}$ and $n_{k}$
that we have $\sigma_{k}^{r}\leq t_{m_{k}}<t_{m_{k}+1}<...<t_{n_{k}}%
<\sigma_{k+1}^{r},$ and hence
\[
\sum_{j=m_{k}}^{n_{k}}\omega_{\mathbf{x}}(t_{j},t_{j+1})\leq\omega
_{\mathbf{x}}(\sigma_{k}^{r},\sigma_{k+1}^{r})+1,\text{ if }k=0,1,...,N_{0}%
^{r}(\mathbf{x,}\left[  0,T\right]  )-1.
\]

To finish we note that
\begin{align*}
\sum_{i=1}^{n}\omega_{\mathbf{x}}\left(  t_{i-1},t_{i}\right)   &  \leq
\sum_{k\in A}\sum_{j=m_{k}}^{n_{k}}\omega_{\mathbf{x}}(t_{j},t_{j+1}%
)+\omega_{\mathbf{x}}(\sigma_{N_{0}^{r}(\mathbf{x,}\left[  0,T\right]  )}%
^{r},\sigma_{T}^{r})\\
&  \leq\sum_{k\in A}\left[  \omega_{\mathbf{x}}(\sigma_{k}^{r},\sigma
_{k+1}^{r})+1\right]  +\omega_{\mathbf{x}}(\sigma_{N_{0}^{r}(\mathbf{x,}%
\left[  0,T\right]  )}^{r},\sigma_{T}^{r})\\
&  \leq\sum_{k=0}^{N_{0}^{r}(\mathbf{x,}\left[  0,T\right]  )-1}\left[
\omega_{\mathbf{x}}(\sigma_{k}^{r},\sigma_{k+1}^{r})+1\right]  +\omega
_{\mathbf{x}}(\sigma_{N_{0}^{r}(\mathbf{x,}\left[  0,T\right]  )}^{r}%
,\sigma_{T}^{r})\\
&  \leq\sum_{k=0}^{N_{0}^{r}(\mathbf{x,}\left[  0,T\right]  )}\omega
_{\mathbf{x}}(\sigma_{k}^{r},\sigma_{k+1}^{r})+N_{0}^{r}(\mathbf{x,}\left[
0,T\right]  ),
\end{align*}

and since the right hand side of the previous estimate no longer depends on
$D,$ we can take the supremum over all $D$ satisfying the constraint in
Definition \ref{accumul}. The conclusion (\ref{key est1}) then follows immediately.
\end{proof}

We are now ready to prove the main result.

\begin{theorem}
\label{main thm}Let $\Lambda\geq1,$ $a\in\Xi(\Lambda).$ Assume that
$\mathbf{X}^{a,x}$ is the $\mathfrak{g-}$valued Markov process, defined on
some probability space, associated with the Dirichlet form
(\ref{Dirichlet_form}). Let $\mathbb{P}^{a,x}$ be the (Borel) probability
measure on $C\left(  [0,\infty),\mathfrak{g}\right)  $ under which the
coordinate process $\mathbf{X}$ has the same law as $\mathbf{X}^{a,x}$, and
let $c_{2},\lambda_{0}\in\left(  0,\infty\right)  $ be the constants in
(\ref{unif}). Assume $p>2,$ and write $M(\mathbf{\cdot},[0,T])$ for the
accumulated local $p-$variation given in Definition \ref{accumul}. Since
$\mathbb{P}^{a,x}$ is supported in $C^{1/p-\text{H\"{o}l}}([0,\infty
),\mathfrak{g}),$ $M(\mathbf{X},[0,T])$ is defined $\mathbb{P}^{a,x}$- almost
surely. Moreover there exist finite $C_1,C_2,C_3>0$, which depend only on $\Lambda$,
$p$, $N$ and $T,$ such that for fixed $r>0$
\begin{equation}
\mathbb{P}^{a,x}\left(  M\left(  \mathbf{X};[0,T]\right)  >R\right)  \leq
\exp(-C_1r^2R^2)+Rr^{-p}C_2\exp(-C_3 Rr^{2-p})\label{main tail}
\end{equation}
for all $R\geq (16\lambda_0c_2^{-1})^{1/2}r^{-2}$. In particular, choosing $r:=R^{-1/p}$ in (\ref{main tail}), yields a better-than-exponential tail for the accumulated local $p$-variation functional, i.e., there exists a finite $C>0$, which depends only on $\Lambda$,
$p$, $N$ and $T,$ such that 
\begin{equation}
\mathbb{P}^{a,x}\left(  M\left(  \mathbf{X};[0,T]\right)  >R\right)  \leq
C\exp\left(  -CR^{2(1-1/p)}\right)  \label{main tail 2}%
\end{equation}
for all $R\geq(16\lambda_{0}c_{2}^{-1})^{p\left(  2p-4\right)  ^{-1}}$.
\end{theorem}

\begin{proof}
We will assume that $T=1$ and write $M\left(  \mathbf{X}\right)  $ and
$N_{0}^{r}\left(  \mathbf{X}\right)  $ in lieu of $M\left(  \mathbf{X;}\left[
0,T\right]  \right)  $ and $N_{0}^{r}\left(  \mathbf{X;}\left[  0,T\right]
\right)  ,$ respectively. The assumption $T=1$ involves no loss of generality
because of the scaling property (\ref{scaling}). The assertion that
$\mathbb{P}^{a,x}$ is supported in $C^{1/p-\text{H\"{o}l}}([0,\infty
),\mathfrak{g})$ is proved in \cite{FrizVictoir2010}, see also our
presentation in Theorem \ref{Holder_bound}.

\qquad\qquad We will prove the main estimate (\ref{main tail}) by using the
family of estimates in Lemma \ref{key est} for different values of $r.$ First
note that it is a straight-forward consequence of lemma \ref{key est} that we
have for any $R>0$ and $r>0$
\[
\left\{  \omega:M\left(  \mathbf{X}\left(  \omega\right)  \right)  >R\right\}
\subset\left\{  \omega:N_{0}^{r}\left(  \mathbf{X}\left(  \omega\right)
\right)  >\frac{R}{2}\right\}  \cup\left\{  \omega:\sum_{j=0}^{N_{0}%
^{r}\left(  \mathbf{X}\left(  \omega\right)  \right)  }\omega_{\mathbf{X}%
\left(  \omega\right)  }\left(  \sigma_{j}^{r},\sigma_{j+1}^{r}\right)
>\frac{R}{2}\right\}  ,
\]
where again $N_{0}^{r}\left(  \mathbf{X}\right)  $ and the sequence $\left(
\sigma_{n}^{r}\right)  _{n=0}^{N_{0}^{r}\left(  \mathbf{X}\right)  }$ are as
given as in (\ref{st}) with $T=1$, and we redefine $\sigma_{N_{0}%
^{r}(\mathbf{x})+1}$ to equal $T=1.$ A simple estimate then gives
\begin{equation}
\mathbb{P}^{a,x}\left(  M\left(  \mathbf{X}\right)  >R\right)  \leq
\mathbb{P}^{a,x}\left(  N_{0}^{r}\left(  \mathbf{X}\right)  >\frac{R}%
{2}\right)  +\mathbb{P}^{a,x}\left(  \sum_{j=0}^{N_{0}^{r}\left(
\mathbf{X}\right)  }\omega_{\mathbf{X}}\left(  \sigma_{j}^{r},\sigma_{j+1}%
^{r}\right)  >\frac{R}{2}\right)  . \label{1}%
\end{equation}
for all $R>0$ and $r>0.$ 

\qquad By Proposition \ref{Gaussian_Tail}
\[
\mathbb{P}^{a,x}\left(  N_{0}^{r}\left(  \mathbf{X}\right)  >\frac{R}%
{2}\right) \leq \exp\left( -\frac{c_2R^2r^2}{16} \right)
\]
for all $R\geq (16\lambda_0c_2^{-1})^{1/2}r^{-2}.$
. It remains to focus on the second term on the right in (\ref{1}). To this end, first note the
following elementary inequality
\[
\omega_{\mathbf{X}}\left(  \sigma_{i}^{r},\sigma_{i+1}^{r}\right)  \leq
\Vert\mathbf{X}\Vert_{1/p-\text{H\"{o}l};[\sigma_{i}^{r},\sigma_{i+1}^{r}%
]}^{p}\left(  \sigma_{i+1}^{r}-\sigma_{i}^{r}\right)  .
\]
Then for any $h>0$ we notice that%
\begin{align*}
\Vert\mathbf{X}\Vert_{1/p-\text{H\"{o}l},[\sigma_{i}^{r},\sigma_{i+1}^{r}%
]}^{p}  &  \leq\sup_{\substack{s\neq t,\left\vert t-s\right\vert \leq
h,\\\left[  s,t\right]  \subset\left[  \sigma_{i}^{r},\sigma_{i+1}^{r}\right]
}}\frac{\Vert\mathbf{X}_{s,t}\Vert_{CC}^{p}}{|t-s|}+\sup_{\substack{s\neq
t,\left\vert t-s\right\vert >h,\\\left[  s,t\right]  \subset\left[  \sigma
_{i}^{r},\sigma_{i+1}^{r}\right]  }}\frac{\Vert\mathbf{X}_{s,t}\Vert_{CC}^{p}%
}{|t-s|}\\
&  \leq\sup_{\substack{s\neq t,\left\vert t-s\right\vert \leq h,\\\left[
s,t\right]  \subset\left[  \sigma_{i}^{r},\sigma_{i+1}^{r}\right]  }%
}\frac{\Vert\mathbf{X}_{s,t}\Vert_{CC}^{p}}{|t-s|}+\frac{(2r)^{p}}{h},
\end{align*}
where the last line follows from the definition of $\sigma_{i}^{r}$ and
$\sigma_{i+1}^{r}$. Using the equality
$\sum_{i=0}^{N_{0}^{r}(\mathbf{X})}(\sigma_{i+1}^{r}-\sigma_{i}^{r})=1$, we
thus have for any $h>0$
\begin{align*}
\sum_{i=0}^{N_{0}^{r}\left(  \mathbf{X}\right)  }\omega_{\mathbf{X}}\left(
\sigma_{i}^{r},\sigma_{i+1}^{r}\right)   &  \leq\sum_{i=0}^{N_{0}^{r}\left(
\mathbf{X}\right)  }\left[  \sup_{\substack{s\neq t,\left\vert t-s\right\vert
\leq h,\\\left[  s,t\right]  \subset\left[  \sigma_{i}^{r},\sigma_{i+1}%
^{r}\right]  }}\frac{\Vert\mathbf{X}_{s,t}\Vert_{CC}^{p}}{|t-s|}\left(
\sigma_{i+1}^{r}-\sigma_{i}^{r}\right)  \right]  +\frac{2^{p}r^p}{h}\\
&  \leq\sup_{\substack{s\neq t,\left\vert t-s\right\vert \leq h,\\\left[
s,t\right]  \subset\left[  0,1\right]  }}\frac{\Vert\mathbf{X}_{s,t}\Vert
_{CC}^{p}}{|t-s|}+\frac{2^{p}r^p}{h}.
\end{align*}
Applying this estimate with the choice $h=2^{p+2}R^{-1}r^p$ we obtain
\[
\sum_{i=0}^{N_{0}\left(  \mathbf{X}\right)  }\omega_{\mathbf{X}}\left(
\sigma_{i},\sigma_{i+1}\right)  \leq\sup_{\substack{s\neq t,\left\vert
t-s\right\vert \leq h,\\\left[  s,t\right]  \subset\left[  0,1\right]  }%
}\frac{\Vert\mathbf{X}_{s,t}\Vert_{CC}^{p}}{|t-s|}+\frac{R}{4}%
\]
and consequently it suffices to bound%
\[
\mathbb{P}^{a,x}\left(  \sup_{\substack{s\neq t,\left\vert t-s\right\vert \leq
h,\\\left[  s,t\right]  \subset\left[  0,1\right]  }}\frac{\Vert
\mathbf{X}_{s,t}\Vert_{CC}^{p}}{|t-s|}\geq\frac{R}{4}\right)  .
\]
To do so, note that if the interval $[s,t]\subseteq\lbrack0,1]$ satisfies
$|t-s|<h$, it must be contained in at least one interval of the form
\[
\left[  (k-1)h,(k+1)h\right]  \quad\text{for some }k=1,...,\left\lceil
h^{-1}\right\rceil .
\]
Therefore,
\begin{equation}
\mathbb{P}^{a,x}\left(  \sup_{\substack{s\neq t,\left\vert t-s\right\vert \leq
h,\\\left[  s,t\right]  \subset\left[  0,1\right]  }}\frac{\Vert
\mathbf{X}_{s,t}\Vert_{CC}^{p}}{|t-s|}\geq\frac{R}{4}\right)  \leq\sum
_{k=1}^{\left\lceil h^{-1}\right\rceil }\mathbb{P}^{a,x}\left(  \sup
_{[s,t]\subseteq\left[  (k-1)h,(k+1)h\right]  }\frac{\Vert\mathbf{X}%
_{s,t}\Vert_{CC}^{p}}{|t-s|}\geq\frac{R}{4}\right)  . \label{sum}%
\end{equation}
We will now show that each term in this sum possesses the desired bound, i.e.,
there exists a positive constant $c>0$ such that
\begin{equation}
\mathbb{P}^{a,x}\left(  \sup_{[s,t]\subseteq\left[  (k-1)h,(k+1)h\right]
}\frac{\Vert\mathbf{X}_{s,t}\Vert_{CC}^{p}}{|t-s|}\geq\frac{R}{4}\right)  \leq
c\exp\left(  -\frac{Rr^{2-p}}{c}\right)  . \label{wts}%
\end{equation}
Because there are only $\lceil h^{-1}\rceil\leq Rr^{-p}$ terms in the sum, it
will follow that we can bound the left hand side of (\ref{sum}) by
\[
Rr^{-p}c\exp\left(  -\frac{1}{c}Rr^{2-p}\right).
\]
 To prove (\ref{wts}) we exploit the scaling property
(\ref{scaling}) with $\tilde{h}:=h^{-1/2}=2^{-1-p/2}R^{1/2}r^{-p/2}$ and the homogeneity of the
$CC$-norm to see that
\[
\left\Vert \mathbf{X}^{a,x}\right\Vert _{1/p-\text{H\"{o}l};[(k-1)h,(k+1)h]}%
^{p}\overset{\mathcal{D}}{=}\frac{1}{\tilde{h}^{p-2}}\left\Vert \mathbf{X}%
^{a^{\tilde{h}},\delta_{\tilde{h}}x}\right\Vert _{1/p-\text{H\"{o}l};[(k-1),(k+1)]}^{p}.
\]
We then conclude with
\begin{align*}
\sup_{y\in\mathfrak{g}}\,\mathbb{P}^{a,y}\left(  \Vert\mathbf{X}%
\Vert_{1/p-\text{H\"{o}l};[(k-1)h,(k+1)h]}^{p}\right.   &  \geq\left.
\frac{R}{4}\right)  =\sup_{y\in\mathfrak{g}}\,\mathbb{P}^{a^{\tilde{h}},y}\left(
\Vert\mathbf{X}\Vert_{1/p-\text{H\"{o}l};[(k-1),(k+1)]}^{p}\geq\frac{R^{p/2}r^{p-p^2/2}
}{2^{p^{2}/2}}\right) \\
&  \leq\sup_{a\in\Xi(\Lambda)}\sup_{y\in\mathfrak{g}}\,\mathbb{P}^{a,y}\left(
\Vert\mathbf{X}\Vert_{1/p-\text{H\"{o}l};[(k-1),(k+1)]}^{p}\geq\frac{R^{p/2}r^{p-p^2/2}
}{2^{p^{2}/2}}\right) \\
&  \leq c_{3}\exp\left(  -\frac{Rr^{2-p}}{c_{3}2^{p}}\right)
\end{align*}
where the last step results from applying Theorem \ref{Holder_bound}.

\end{proof}

\bibliographystyle{plain}
\bibliography{/Users/marcelogrodnik/Documents/library}

\end{document}